\documentclass{amsart}
\usepackage{amsmath,amsthm,amssymb}
\usepackage{colortbl}
\usepackage{mathabx}
\usepackage{mathrsfs} 
\usepackage[all]{xy}
\usepackage{graphicx}
\usepackage{float}
\theoremstyle{plain}
\title[On the linearity of origin-preserving automorphisms]{On the linearity of origin-preserving automorphisms of quasi-circular domains in $\mathbb C^n$}
\author{Atsushi Yamamori}
\address{The Center for Geometry and its Applications, Pohang University of Science and Technology, Pohang 790-784, Republic of Korea}
\theoremstyle{definition}
\thanks{The research of the author was supported in part by SRC-GaiA (Center for Geometry and its
Applications), the Grant 2011-0030044 from The Ministry of Education, The Republic of Korea}
\email{ats.yamamori@gmail.com, yamamori@postech.ac.kr}
\newtheorem{theorem}{Theorem}[section]

\newtheorem{proposition}[theorem]{Proposition}
\newtheorem{corollary}{Corollary}[section]
\newtheorem{remark}{Remark}
\newtheorem{problem}{Problem}
\newtheorem{example}{Example}[section]
\newtheorem{definition}{Definition}[section]

\begin{document}
\begin{abstract}
A theorem due to Cartan asserts that every origin-preserving automorphism of bounded circular domains with respect to the origin is linear.
In the present paper, by employing the theory of Bergman's representative domain, we prove that under certain circumstances
Cartan's assertion remains true for quasi-circular domains in $\mathbb C^n$. 
Our main result is applied to obtain some simple criterions for the case $n=3$ and to prove that Braun-Kaup-Upmeier's theorem remains true
for our class of quasi-circular domains.
\end{abstract}
\maketitle
\section{Introduction}
\subsection{Backgrounds}
Throughout this article, we assume that domains are bounded and contain the origin.
Let $m_1, \ldots, m_n$ be positive integers such that $\mathrm{gcd}(m_1,\ldots, m_n)=1$ and $m_1 \leq \cdots \leq m_n$.
\begin{definition}
A domain $D \subset \mathbb C^n$ is a called quasi-circular domain with weight $(m_1, \ldots, m_n)$ (or a $(m_1, \ldots, m_n)$-circular domain) if it is invariant under
\[
f_{m,\theta}: D \rightarrow \mathbb C^n, \quad (z_1, \ldots, z_n) \mapsto (e^{i m_1 \theta} z_1 , \ldots , e^{i m_n \theta} z_n).
\]
In particular, if $m_1=\cdots=m_n=1$, then it is called circular.
\end{definition}
Let us denote the isotropy subgroup of the holomorphic automorphism group $\mathrm{Aut}(D)$ by $\mathrm{Iso}_p(D)$:
\[
\mathrm{Iso}_p(D):=\{f \in \mathrm{Aut}(D): f(p)=p \}.
\]
There is a remarkable result due to Cartan concerning the structure of $\mathrm{Iso}_p(D)$.
\begin{theorem}\label{thm1}
If a domain $D$ is circular, then every automorphism $f \in \mathrm{Iso}_0(D)$ is linear (i.e. $\mathrm{Iso}_0(D) \subset GL(D)$).
\end{theorem}
Although the quasi-circular domains have less symmetry than the circular domains have, it is known that there is an analogue of Cartan's result.
Indeed, it is proved by Kaup \cite{Kaup} that every automorphism $f \in \mathrm{Iso}_0(D)$ of a quasi-circular domain $D$ is a polynomial mapping.
From an analogue of Cartan's theorem, it is tempting to ask the following problem:
\begin{problem}\label{prob1}
Find a class of weights of quasi-circular domains such that Cartan's assertion remains true.
\end{problem}
In \cite{Yamamori-Bull},
it is proved that $\mathrm{Iso}_0(D) \subset GL(D)$ for quasi-circular domains in $\mathbb C^2$ such that $2 \leq m_1 \leq m_2$.
This answered the above problem for $n=2$.
In this article, we continue our previous work.
More precisely, we introduce a certain class of weights $\mathscr L_n$ for quasi-circular domains in $\mathbb C^n$
and prove the linearity of origin-preserving automorphisms.
\begin{remark}
After the main part of this work was finished,
the author came to know about Rong's preprint \cite{Rong}.
His result states that if $f=(f_1,\ldots, f_n)$ is an origin-preserving automorphism of a quasi-circular domain in $\mathbb C^n$, then the degree of $f$ is less than or equal to the resonance order.
Therefore our problem is equivalent to find a class of weights which implies the resonance order $\mu=1$.
\end{remark}
\begin{remark}
In Bell's paper \cite{Bell} and Berteloot-Patrizio's paper \cite{Ber}, Cartan theorem was generalized  for proper mappings between bounded circular domains in $\mathbb C^n$
\end{remark}
\subsection{Comments on our approach}
As is well-known, Theorem \ref{thm1} is usually shown as a consequence of the following theorem, which is also due to Cartan:
\begin{theorem}[Cartan Uniqueness Theorem]
Let $D$ be a bounded domain and $f: D \rightarrow D$ be holomorphic such that $f(p)=p$ for some $p \in D$ and the Jacobian matrix of $f$ at $p$ is the identity matrix (i.e. $\mathrm{Jac}(f, p)=\mathrm{id}$).
Then $f$ is the identity mapping of $D$.
\end{theorem}
This theorem also holds when $D$ is hyperbolic (cf.~\cite{Kob}). However we face the difficulties to prove this kinds of theorems without boundedness or hyperbolicity.
Thus, to determine whether or not Theorem \ref{thm1} holds for a given non-hyperbolic domain is nontrivial question.

A notable aspect of the importance of our approach is that one can bypass this difficulty.
Indeed, it is observed in our previous paper \cite{KNY} that
\emph{Theorem \ref{thm1} remains true for any circular domains whenever ``the Bergman mapping" is well-defined.}
Although it is a simple observation from Ishi-Kai's paper \cite{IK}, it appears not to have been noticed by many mathematicians.
Moreover this observation gives us an unbounded non-hyperbolic example $D_{n,m}$ for which Theorem \ref{thm1} remains true:
\[
D_{n,m}=\{(z,\zeta) \in\mathbb C^n \times \mathbb C^m: \|\zeta\|^2 < e^{-\mu \|z\|^2}\}, \quad \mu>0.
\]
Here $D_{n,m}$ is called the Fock-Bargmann-Hartogs domain in \cite{Y2013}.
The above observation is applied to the following studies:
\begin{itemize}
\item an explicit description of the automorphism group of $D_{n,m}$ \cite{KNY},
\item rigidity properties of proper holomorphic mappings for $D_{n,m}$ \cite{TW}.
\end{itemize}
It would be desirable to give a non-hyperbolic quasi-circular example for which Cartan theorem holds and the automorphism group is explicitly computed.
Unfortunately, we do not have any such examples at the time of writing this paper.
To find a such example is interesting problem and it will be investigated in the future research.
\subsection{Outline of this paper}
Before moving to the next section, let us pause to explain an outline of our argument.
\begin{itemize}
\item[(i)] By using the theory of representative domain (Proposition \ref{pro2}), we reduce Problem \ref{prob1} to the study of ``the Bergman metric tensor" $T_D$.
More precisely, it is reduced to find a class of weights such that $T_D(z,0) \equiv T_D(0,0)$ (Problem \ref{prob}).
\item[(ii)] In Section \ref{mainsec}, we introduce a class of weights $\mathscr L_n$ for quasi-circular domains in $\mathbb C^n$.
\item[(iii)] 
For any quasi-circular domains, it is proved that the upper triangular entries of $T_D(z,0)$ are constants. To show that the remaining entries are constants, 
we use some properties of the class $\mathscr L_n$.
\item[(iv)] By (i) and (iii), it is proved that Cartan's theorem remains true for our class of quasi-circular domains. It is known by Kaup that
if two quasi-circular domains are biholomorphic then there is a biholomorphism fixing the origin.
Using this fact, we also prove that Braun-Kaup-Upmeier's theorem also remains true for our cases.
\end{itemize}
As an application of our theorem, three-dimensional cases are further studied.
\section{Definitions and basic facts}
Let us begin our study with basic definitions.
Let $D\subset \mathbb C^n$ be a domain and $A^2 (D)$ be the space of square integrable holomorphic functions on $D$.
The space $A^2(D)$ is called the Bergman space of $D$ and its reproducing kernel $K_D$ is called the Bergman kernel.
Now we introduce an important class of domains, which is so-called minimal domains (see also \cite[Theorem 3.1]{Mas}).
\begin{definition}
A domain $D$ is called minimal if there is a point $z_0$ such that $K(z,z_0) \equiv K(z_0,z_0)$ for any $z \in D$. 
The point $z_0$ is called the center of a minimal domain $D$.
\end{definition}

Define an $n\times n$ matrix $T_D(z,w)$ by
$$T_D (z,w):=
\begin{pmatrix} 
\dfrac{\partial^2 }{\partial \overline{w_1}\partial z_1}\log K_D(z,w) & \cdots & \dfrac{\partial^2 }{\partial \overline{w_1}\partial z_n}\log K_D(z,w)
\\ \vdots & \ddots & \vdots\\
\dfrac{\partial^2 }{\partial \overline{w_n}\partial z_1}\log K_D(z,w) & \cdots & \dfrac{\partial^2 }{\partial\overline{ w_n}\partial z_n}\log K_D(z,w)
\end{pmatrix},$$
for $z,w \in D$ such that $K_D(z,w)\not =0$. The matrix $T_D(z,z)$ is positive definite for any $z \in D$ when $D$ is bounded.
In the following, we denote for simplicity $$K_{\overline{i}{j}} (z,w)=\frac{\partial^2 \log K_D (z,w)}{\partial \overline{w_i}   \partial z_j} .$$
We next introduce another important class of domains by using the matrix $T_D$ (see also \cite[Theorem 2.2]{Tsu}).
\begin{definition}
A bounded domain $D$ is called a representative domain if there is a point $z_0 \in D$ such that
$T_D(z,z_0) \equiv T_D(z_0,z_0)$.
The point $z_0$ is called the center of a representative domain $D$.
\end{definition}
These two important class of domains have been studied by various authors (see \cite{IY}, \cite{Mas}, \cite{OK}, \cite{Tsu} and references therein).
Let us prepare some facts on these class of domains.
 The first fact is the minimality of the quasi-circular domains.
In \cite{Yamamori-Bull}, we proved this proposition for $n=2$.
Although the same proof works for any $n\in \mathbb Z+$, we give a proof for the sake of completeness.
\begin{proposition}\label{pro1}
Every bounded quasi-circular domain in $ \mathbb C^n $ is minimal with the center at the origin.
\end{proposition}
\begin{proof}
Let us first recall the transformation formula of the Bergman kernel under any biholomorphism $f$ between two domains $D, D'$:
\[
K_D(z,w)= \det \mathrm{Jac}(f,z) K_D(f(z), f(w)) \overline{\det \mathrm{Jac}(f,w)}.
\]
Applying this formula to $D=D'$ and $f=f_{m,\theta}$, we obtain
\[
K_D(z,0)=K_D(f_{m,\theta},0).
\]
Put $K_D(z,0)=\sum_{k \geq 0}a_k z_1^{k_1}\cdots z_n^{k_n}$. Then it is equivalent to
\[
a_k = e^{i (\sum_{j=1}^n m_j k_j ) \theta} a_k.
\]
Since $\sum_{j=1}^n m_j k_j  \neq 0$ unless $(k_1,\ldots, k_n)=(0,\ldots, 0)$, we conclude that $a_k=0$ except the constant term and thus $K_D(z,0)\equiv K_D(0,0)$ as desired.
\end{proof}
The next assertion states the linearity of origin-preserving biholomorphisms between minimal representative domains.
\begin{proposition}\label{pro2}
If $D, D' \subset \mathbb C^n$ are bounded minimal representative domains with the center at the origin, Then any biholomorphism which maps the center of $D$ to that of $D'$ is linear.
\end{proposition}
This proposition follows from the following properties of the Bergman mapping
$\sigma_{0}^D (z)=T_D(0,0)^{-1/2}\mbox{grad}_{\overline{w}}\log \frac{K_D(z,w)}{K_D(0,w) } |_{w=0}.$
\begin{itemize}
\item the Bergman mapping $\sigma_{0}^D$ is linear and invertible if $D$ is a bounded representative domain.
\item If $f: D \rightarrow D'$ is an origin-preserving biholomorphism of two minimal representative domains with the center at the origin, 
then we have $\sigma_{0}^{D'} \circ f = U \circ \sigma_{0}^{D}$ for some unitary transformation $U: \mathbb C^n \rightarrow \mathbb C^n$.
\end{itemize}
For details of these facts, we refer the reader to \cite{IK} and \cite{Yamamori-Bull}.\par
For instance, every bounded circular domain is a minimal representative domain with the center at the origin.
This fact, together with Proposition \ref{pro2}, gives us another proof of Theorem \ref{thm1} without Cartan Uniqueness Theorem (cf. \cite{IK}).\par
On the other hand, quasi-circular domains are not representative domains in general. 
Thanks to these propositions, our problem is reduced to the following problem:
\begin{problem}\label{prob}
Find a class of weights of quasi-circular domains $D$ such that $T_D(z,0) \equiv T_D(0,0)$ for any $z \in D$.
\end{problem}
We conclude this section with a remark on $T_D$.
\begin{remark}
As we proved in Proposition \ref{pro1}, every bounded quasi-circular domain $D$ is minimal with the center at the origin.
In other words, the Bergman kernel $K_D$ has a property $K_D(z,0)\equiv K_D(0,0)$. Moreover it is well-known that $K_D(0,0) >0$.
Thus the Bergman kernel $K_D(z,0)$ is zero-free for any $z \in D$ and $T_D(z,0)$ is well-defined for any $z \in D$.
\end{remark}
\section{Main results}
\subsection{Quasi-circular domains in $\mathbb C^n$}\label{mainsec}
In the following, for simplicity of our argument, we always assume that the weight $(m_1,\ldots, m_n)$ of a quasi-circular domain satisfies
$m_1 < \cdots < m_n$.
Let us define the set $I_{m,M}^\ell$ by
\begin{align*}
I_{m,M}^\ell &:=\left\lbrace r_{k,m,i}^\ell: 
\begin{array}{c}
1\leq i\leq\ell, (k_1,\ldots, k_\ell) \neq0 \mbox{ such that } \\
M\sum_{q=1}^{\ell} m_q>  r_{k,m,i}^\ell > (M-1)\sum_{q=1}^{\ell} m_q 
\end{array}
\right\rbrace,\\
r_{k,m,i}^\ell&:=m_i+ \sum_{q=1}^{\ell} m_q k_q, 
\end{align*}
where $M \in \mathbb Z+$ and $k_1,\ldots ,k_{\ell} \in \mathbb Z_{\geq 0}$.
The matrices
$(r_{k,m,1}^2)_{0 \leq k_1, k_2 \leq 2}$ and $(r_{k,m,2}^2)_{0 \leq k_1, k_2 \leq 2}$ are given as follows:
\begin{align*}
(r_{k,m,1}^2)_{0 \leq k_1, k_2 \leq 2}&=
\left(
\begin{array}{ccc}
 m_1 & m_1+m_2 & m_1+2 m_2 \\
 2 m_1 & 2 m_1+m_2 & 2 m_1+2 m_2
   \\
 3 m_1 & 3 m_1+m_2 & 3 m_1+2 m_2
\end{array}
\right),\\
(r_{k,m,2}^2)_{0 \leq k_1, k_2 \leq 2}&=
\left(
\begin{array}{ccc}
 m_2 & 2 m_2 & 3 m_2 \\
 m_1+m_2 & m_1+2 m_2 & m_1+3
   m_2 \\
 2 m_1+m_2 & 2 m_1+2 m_2 & 2 m_1+3
   m_2
\end{array}
\right).
\end{align*}
\begin{definition}
Let $(m_1,\ldots, m_n)$ be the weight of a quasi-circular domain $D$ in $\mathbb C^n$ with $n \geq 2$.
Define the set $\mathscr L_2$ by $\mathscr L_2:=\{(m_1,m_2): 2 \leq m_1, m_2 \not \equiv 0(\mbox{mod $m_1$})\}$. 
For $n \geq 3$, $\mathscr L_n$ is defined as the set of positive integers $(m_1, \ldots, m_n)$ such that
\begin{itemize}
\item[(i)] $(m_1, \ldots, m_{n-1})$ is an element of $\mathscr L_{n-1}$,
\item[(ii)] there exists an integer $M_n \geq 1$ such that  $m_n \not\in I_{m,M_n}^{n-1}$ and
\[
 M_n\sum_{q=1}^{n-1} m_q > m_n > (M_n-1)\sum_{q=1}^{n-1} m_q.
\]
\end{itemize}
Here $M_n$ depends only on $m_n$.
\end{definition}
If $n=2$, then the weight $(m_1, m_2)$ of a quasi-circular domain $D \subset \mathbb C^2$ satisfies $\mathrm{gcd}(m_1,m_2)=1$ by definition.
Thus the condition ``$m_2 \not \equiv 0(\mbox{mod $m_1$})$" is superfluous in this case.
If $n=3$, $(m_1,m_2,m_3)$ is in $\mathscr L_3$ if and only if:
\begin{enumerate}
\item[(a)] $m_1 \geq 2$ and $m_2 \not \equiv 0 (\mbox{mod $m_1$})$,
\item[(b)] there exists an integer $N \geq 1$ such that $m_3 \not\in I_{m,N}^2$ and $$N(m_1+m_2) > m_3 > (N-1)(m_1+m_2).$$
\end{enumerate}
We note that if $N=1$, then the set $I_{m,1}^2$ is given by $$I_{m,1}^2=\{n m_1: n \geq 2, nm_1 < m_1+m_2\}.$$
Thus the condition (b) is equivalent to the following:
\begin{itemize}
\item[(b')] $m_1+m_2>m_3 $ and $m_3 \not \equiv 0 (\mbox{mod $m_1$})$.
\end{itemize}
Let us give some examples.
\begin{example}
The weight of the form $(2, m_2, m_3)$ cannot satisfy (a),(b').
Indeed, if $m_2$ is a odd number, then $m_3$ must be an even number by the condition $m_2+2 > m_3 > m_2$.
\end{example}
\begin{example}\label{ex}
Let us consider the weights $(m_1,m_2,m_3)=(3,5,7), (4,5,7)$.
For these weights, we can easily check the conditions (a),(b').
The followings are quasi-circular domains with these weights.
\begin{align*}
D_1&=\{(z_1,z_2,z_3)\in\mathbb B^3: |z_1^{35}+ z_2^{21} + z_3^{15} |<1\},\\
D_2&=\{(z_1,z_2,z_3)\in\mathbb B^3: |z_1^{35}+ z_2^{28} + z_3^{20} |<1\}.
\end{align*}
\end{example}
\begin{example}
If $(m_1,m_2,m_3)=(3,7,11)$, the set $I_{m,2}^2$ is given by $$I_{m,2}^2=\{12,13,14,15,16,17,18,19\}.$$
Thus this weight satisfies the conditions (a), (b) with $N=2$.
In this case, $r_{k,m,1}^2$ and $r_{k,m,2}^2$ are given as follows.
\begin{align*}
(r_{k,m,1}^2)_{\substack{0 \leq k_1 \leq 6\\ 0 \leq k_2 \leq 3}}=
\left(
\begin{array}{cccc}
 3 & 10 & \cellcolor[gray]{0.8}17  &24\\
 6 & \cellcolor[gray]{0.8}13 & 20 & 27\\
 9 & \cellcolor[gray]{0.8}16 & 23  &30\\
\cellcolor[gray]{0.8} 12 & \cellcolor[gray]{0.8}19 & 26 &33\\
\cellcolor[gray]{0.8} 15 & 22 & 29  & 36\\
\cellcolor[gray]{0.8} 18 & 25 & 32 &  39\\
21 & 28 & 35 & 42
\end{array}
\right), \quad
(r_{k,m,2}^2)_{\substack{0 \leq k_1 \leq 5\\ 0 \leq k_2 \leq 2}}=
\left(
\begin{array}{ccc}
 7 & \cellcolor[gray]{0.8}14 & 21 \\
 10 & \cellcolor[gray]{0.8}17 & 24 \\
 \cellcolor[gray]{0.8}13 & 20 & 27 \\
\cellcolor[gray]{0.8} 16 & 23 & 30 \\
\cellcolor[gray]{0.8} 19 & 26 & 33\\
22 & 29 & 36
\end{array}
\right).
\end{align*}
\end{example}
\begin{example}\label{ex35}
In the case of the weight of the form $(3, 5, m_3)$, there is the integer $m_3=7$ such that $(3, 5, m_3) \in \mathscr L_3$ with $N=1$.
On the other hand, there does not exist an integer $m_3$ such that $(3, 5, m_3) \in \mathscr L_3$ with $N=2$.
Indeed, $I_{m,2}^2$ is given by $$I_{m,2}^2=\{9,10,11,12,13,14,15\}.$$
In this case, $r_{k,m,1}^2$ and $r_{k,m,2}^2$ are given as follows.
\begin{align*}
(r_{k,m,1}^2)_{\substack{0 \leq k_1 \leq 5\\ 0 \leq k_2 \leq 3}}=
\left(
\begin{array}{cccc}
 3 & 8 & \cellcolor[gray]{0.8}13 & 18 \\
 6 & \cellcolor[gray]{0.8}11 & 16  & 21\\
\cellcolor[gray]{0.8} 9 & \cellcolor[gray]{0.8}14 & 19  & 24\\
 \cellcolor[gray]{0.8}12 & 17 & 22 & 27\\
 \cellcolor[gray]{0.8}15 & 20 & 25 & 30\\
18 & 23 &28 & 33
\end{array}
\right), \quad
(r_{k,m,2}^2)_{\substack{0 \leq k_1 \leq 4\\ 0 \leq k_2 \leq 3}}=
\left(
\begin{array}{cccc}
 5 & \cellcolor[gray]{0.8}10 & \cellcolor[gray]{0.8}15 & 20\\
 8 & \cellcolor[gray]{0.8}13 & 18 & 23\\
 \cellcolor[gray]{0.8}11 & 16 & 21 & 26\\
\cellcolor[gray]{0.8} 14 & 19 & 24 & 29\\
17 & 22 & 27 & 32
\end{array}
\right).
\end{align*}
Since $I_{m,N}^2$ is a finite set and the mapping $\tau$
\[
\tau : I_{m,N}^2 \rightarrow I_{m,N+1}^2, \quad r_{k,m,i}^\ell \mapsto r_{k,m,i}^\ell+ m_1+m_2,
\]
is injective, we see that $\# I_{m,N}^2 \leq \# I_{m,N+1}^2$. This implies that 
there does not exist an integer $m_3$ such that $(3, 5, m_3) \in \mathscr L_3$ for any $N\geq2$.
\end{example}
From these examples, it is appropriate to ask the following problem.
\begin{problem}
How does the weight of a quasi-circular domain affect the cardinal number of $I_{m,2}^2$ ?
\end{problem}
This problem will be studied in Section \ref{C3}. 
Now let us prove the following theorem which gives an answer to Problem \ref{prob1}.
\begin{theorem}
Let $D, D' $ be bounded quasi-circular domains in $\mathbb C^n$ such that whose weights are in $\mathscr L_n$, 
then every biholomorphism $f: D \rightarrow D'$ with $f(0)=0$ is linear.
\end{theorem} 
\begin{proof}
By Propositions \ref{pro1} and \ref{pro2}, it is enough to show that $T_D(z,0)\equiv T_D(0,0)$.
To this end, let us first recall the transformation formula of $T_D$, which is analogous to that of the Bergman kernel $K_D$:
\[
T_{D_1}(z,w)=\mathrm{Jac}(g,z) T_{D_2}(g(z), g(w)) \overline{{}^t \mathrm{Jac}(g,w)},
\]
where $g: D_1 \rightarrow D_2$ is a biholomorphism between two bounded domains.
Applying this formula to $D_1=D_2=D$ and $g=f_{m,\theta}$, we obtain
\begin{multline*}
\left( K_{\overline{i}j} (z,0) \right)_{1\leq i,j \leq n}
\\=
\mathrm{diag}(e^{i m_1 \theta}, \ldots, e^{i m_n \theta})
\left( K_{\overline{i}j} (f_{m,\theta}(z),0) \right)_{1\leq i,j \leq n}
\mathrm{diag}(e^{-i m_1 \theta}, \ldots, e^{-i m_n \theta}).
\end{multline*}
Comparing each entry, we have
\begin{align}
&K_{\overline{i}i} (z,0)=K_{\overline{i}i} (f_{m,\theta}(z),0), \quad \mbox{for $1\leq i \leq n $},\label{1}\\
&K_{\overline{i}j} (z,0)=e^{i(m_j-m_i) \theta} K_{\overline{i}j} (f_{m,\theta}(z),0),\quad \mbox{for $1\leq i < j \leq n $},\label{2}\\
&K_{\overline{j}i} (z,0)=e^{i(m_i-m_j) \theta} K_{\overline{j}i} (f_{m,\theta}(z),0),\quad \mbox{for $1\leq i < j \leq n $}\label{3}.
\end{align}
Let us check that these relations imply our conclusion.\par
Using a similar argument of Proposition \ref{pro1}, we see that
$K_{\overline{i}i} (z, 0) \equiv K_{\overline{i}i}(0,0) $ 
for any $1\leq i \leq n $.\par
Consider the relation \eqref{2}. Putting $K_{\overline{i}j} (z,0)=\sum a_k z^k$, we see that \eqref{2} is equivalent to
\[
a_k = e^{i \{(m_j-m_i)+ \sum_{r=1}^n m_r k_r  \} \theta} a_k.
\]
The origin is the unique fixed point of the rotation mapping $r_\theta: z \mapsto e^{i \theta } z$ for $\theta \neq 0$.
By definition, we see that $m_j-m_i >0$ and thus $m_j-m_i+ \sum_{r=1}^n m_r k_r >0$ for any $1\leq i < j \leq n $. These facts tell us that
$K_{\overline{i}j} (z, 0) \equiv K_{\overline{i}j}(0,0) $ for any $1\leq i < j \leq n $.
\par
Let us consider the remaining case \eqref{3}. Using again the Taylor expansion, we put $K_{\overline{j}i} (z,0)=\sum a_k z^k$.
Then the relation \eqref{3} implies that
\[
a_k = e^{i \{(m_i-m_j)+ \sum_{r=1}^n m_r k_r  \} \theta} a_k.
\]
Moreover it is easy to see that the exponent $c_{k,m}^{(i,j)}=(m_i-m_j)+ \sum_{r=1}^n m_r k_r$ is a non-zero constant for any $k_j \geq 1, k_\ell \geq 0, 1 \leq \ell \leq n, \ell \neq j$.\par
Putting $k_2=0$, we see that the condition $c_{k,m}^{(1,2)}=0$ is equivalent to 
\begin{align}\label{cond}
 m_1(k_1+1) + m_3 k_3 + \cdots + m_n k_n =m_2.
\end{align}
The assumption $m_1 < m_2 <\cdots < m_n$ gives us $k_3=\cdots=k_n=0$. Thus \eqref{cond} is reduced to
\[
m_1(k_1+1)= m_2.
\]
By the condition $(m_1,m_2) \in \mathscr {L}_{2}$ we know that $m_1(k_1+1)\neq m_2$ for any $k_1 \geq 0$.
Applying a similar argument to $c_{k,m}^{(i,j)}$, we find that $c_{k,m}^{(i,j)}=0$ is equivalent to
\[
m_i + \sum_{r=1}^{j-1} m_r k_r = m_j.
\]
In other words, it is equivalent to $r_{k,m,i}^{j-1}=m_j.$
Since $m_j$ satisfies inequalities:
$$M_j\sum_{q=1}^{j-1} m_q > m_j > (M_j-1)\sum_{q=1}^{j-1} m_q,$$
we know that $m_j \in I_{m,M_j}^{j-1}$ if $c_{k,m}^{(i,j)}=0$.
On the other hand, by the condition $(m_1,\ldots, m_j) \in \mathscr L_{j}$, we have $m_j \not \in I_{m,M_j}^{j-1}$.
Thus we obtain $c_{k,m}^{(i,j)}\neq0$ for any $1 \leq i< j \leq n$.
Hence we conclude that $T_D(z,0)$ is a constant matrix.
\end{proof}
\begin{remark}
In the proof of the theorem,
The assumption of $\mathscr L_n$ is used only for the case \eqref{3}.
In other words, the upper triangular entries of $T_D (z,0)$ are constants for any bounded quasi-circular domains\par
We also remark that the exponent $m_j-m_i+ \sum_{r=1}^n m_r k_r >0$ is a non-zero integer when $k_1= \cdots = k_n=0$.
Thus it follows that $K_{\overline{i}j} (z,0) \equiv K_{\overline{i}j} (0,0)\equiv 0$ for $1\leq i < j \leq n $.
\end{remark}
It is known that if two bounded quasi-circular domains $D$ and $D'$ are biholomorphic then there is a biholomorphic mapping $f$ such that $f(0)=0$ (see the proof of Folgerung 1 in Kaup's paper \cite{Kaup}).
This fact and our main result give us the following corollary:
\begin{corollary}\label{cor1}
Let $D, D' \subset \mathbb C^n$ be bounded quasi-circular domains whose weights are in $\mathscr L_n$.
Then they are biholomorphic if and only if they are linearly equivalent.
\end{corollary}
This corollary states that the following theorem due to Braun, Kaup and Upmeier remains true for our cases (cf. \cite{BKU}):
\begin{theorem}\label{BKUthm}
Let $D, D' \subset \mathbb C^n$ be circular domains.
Then they are biholomorphic if and only if they are linearly equivalent.
\end{theorem}

\subsection{Quasi-circular domains in $\mathbb C^3$}\label{C3}
Let us turn to the case $n=3$.
By using the conditions (a), (b'), we obtain the following criterion for the linearity.
\begin{corollary}\label{cor}
Let $D$ be a bounded quasi-circular domain in $\mathbb C^3$ and $(m_1,m_2, m_3)$ its weight with the conditions
\begin{itemize}
\item[(i)] $3\leq m_1$,
\item[(ii)] $m_2, m_3 \not\equiv 0  (\mbox{mod } m_1)$,
\item[(iii)] $m_1 + m_2 > m_3$.
\end{itemize}
Then every origin-preserving automorphism of $D$ is linear.
\end{corollary}
If integers $m_2$ and $m_3$ are both prime numbers, then a weight $(m_1,m_2,m_3)$ satisfies (ii) of the above corollary for any $3 \leq m_1 < m_2$.
This simple observation gives us the next corollary.
\begin{corollary}
Let $D$ be a bounded quasi-circular domain in $\mathbb C^3$ and $(m,p_1, p_2)$ its weight where $p_1, p_2 $ are odd prime numbers such that $5 \leq p_1 < p_2$.
Then every origin-preserving automorphism of $D$ is linear for any number $m$ such that $p_2-p_1< m< p_1$.
\end{corollary}
In particular, if $(p_1,p_2)$ is a twin prime pair, then we have the following:
\begin{corollary}
Let $D$ be a bounded quasi-circular domain in $\mathbb C^3$ and $(m,p_1, p_2)$ its weight where $(p_1, p_2) \neq (3,5)$ is a twin prime pair.
Then every origin-preserving automorphism of $D$ is linear for any number $m$ such that $3\leq  m< p_1$.
\end{corollary}
If $3\leq m_1< m_2 < m_3 < 2m_1$, then the weight $(m_1,m_2,m_3)$ satisfies the assumptions (i), (ii), (iii) of Corollary \ref{cor}.
Thus we have the following (cf. \cite[pp.264]{Kaup}).
\begin{corollary}
Let $D$ be a bounded quasi-circular domain in $\mathbb C^3$ and $(m_1,m_2,m_3)$ its weight with $3\leq m_1< m_2 < m_3 < 2m_1$.
Then every origin-preserving automorphism of $D$ is linear.
\end{corollary}
The next theorem tells us how many integers in the interval $(m_1+m_2, 2 (m_1+m_2) )$ satisfy $(m_1,m_2, m_3) \in \mathscr L_3$ for given numbers $m_1,m_2$.
In the following, we denote, as usual, by $\lfloor \mbox{ } \rfloor$ and $\lceil \mbox{ } \rceil$ the floor function and the ceil function respectively.
\begin{theorem}
Let $m_1,m_2$ be prime numbers such that $5 \leq m_1 < m_2 $. Then we have
\begin{itemize}
\item[(i)]
if $2m_1 <  m_2$, then
there exists $d=m_1+m_2-5-\lfloor \frac{2m_2}{m_1} \rfloor$ numbers $\{a_i\}_{i=1}^d$ in the interval $(m_1+m_2, 2 (m_1+m_2) )$ 
such that $(m_1,m_2, a_i)\in \mathscr L_3$
for any $1 \leq i \leq d$,
\item[(ii)]
if $2m_1> m_2$, then
there exists $d'=m_1+m_2-6-\lfloor \frac{2m_2}{m_1} \rfloor$ numbers $\{b_i\}_{i=1}^{d'}$ 
in the interval $(m_1+m_2, 2 (m_1+m_2)  )$
such that $(m_1,m_2, b_i)\in \mathscr L_3$
for any $1 \leq i \leq d'$.
\end{itemize}
\end{theorem}
\begin{proof}
We first observe that
\[
\{r_{k,m,1}^2: k_1, k_2, \geq 0\}= \bigcup_{k_2=0}^\infty \{rm_1 +k_2m_2 : r \geq 1\}.
\]
It implies that $0 \leq k_2 \leq 2$ if $r_{k,m,1}^2 \in I_{m,2}^2$. In the case of $k_2=2$, except the element $m_1 +2m_2$,
we have $rm_1 +2m_2 \geq 2(m_1+m_2)$.
Therefore $r_{k,m,1}^2$ is an element of $I_{m,2}^2$ if and only if $r_{k,m,1}^2 \in S_1 \cup S_2 \cup S_3$,
where $S_1, S_2, S_3$ are given by
\begin{align*}
S_1&=\{rm_1 : r \geq 2, 2(m_1+m_2) > rm_1 > m_1+m_2\},\\
S_2&=\{rm_1 +m_2: r \geq 2, 2(m_1+m_2) > rm_1 +m_2 > m_1+m_2\},\\
S_3&=\{m_1+2m_2\}.
\end{align*}
By a similar argument, we can show that $r_{k,m,2}^2$ is an element of $I_{m,2}^2$ if and only if $r_{k,m,2}^2 \in S_2 \cup S_3 \cup S_4$,
where 
\begin{align*}
S_4&=\{rm_2 : 2\leq r \leq 3, 2(m_1+m_2) > rm_2 > m_1+m_2\}.
\end{align*} 
Since $2(m_1+m_2) > m_1+2m_2$ and $2m_2 > m_1+m_2$ for any positive integers $m_1, m_2$, we have $\# S_3=1$ and $1 \leq \# S_4 \leq 2$.
Now we claim that $S_i \cap S_i = \emptyset$ for any $1 \leq i,j \leq 4$. 
Here we check only $S_1 \cap S_4 = \emptyset$ and omit details of remaining cases.
If $S_1 \cap S_4 \not= \emptyset$, then there exists $r \geq 2$ such that $rm_1= 2m_2$ or  $rm_1= 3m_2$.
Since the right hand side of $rm_1= 2m_2$ is an even number and $m_1$ is an odd prime number, it is equivalent to $r' m_1 =m_2$ for some integer $r'$.
In particular, we see that $r' \neq 1$ by $m_1 < m_2$.
It contradicts that $m_2$ is a prime number. We next consider the case $rm_1= 3m_2$. In this case, using the condition $m_1 \geq 5$, we can conclude that $rm_1\not= 3m_2$ with the same logic.
It implies that $S_1 \cap S_4 = \emptyset$.
Thus we have $I_{m,2}^2 = \bigcup_{i=1}^4 S_i$ and 
$\# I_{m,2}^2 = \sum_{i=1}^4 \# S_i$. Then $\# S_1, \# S_2$ and $\# S_4$ are given by
\begin{align*}
\# S_1 &= \left\lfloor \dfrac{2(m_1+m_2)}{m_1} \right\rfloor - \left\lceil \dfrac{m_1+m_2}{m_1} \right\rceil +1,\\
&=2+ \left\lfloor \dfrac{2m_2}{m_1} \right\rfloor - \left(1+\left\lceil \dfrac{m_2}{m_1} \right\rceil\right)+1,\\
&= 2+\left\lfloor \dfrac{2m_2}{m_1} \right\rfloor -\left\lceil \dfrac{m_2}{m_1} \right\rceil,\\
\# S_2 &= \left\lfloor \dfrac{m_1+m_2}{m_1} \right\rfloor=1+  \left\lfloor \dfrac{m_2}{m_1} \right\rfloor,\\
\# S_4 &=
\begin{cases}
1 & \text{if $2m_1 < m_2$ },\\
2 & \text{if $2m_1 > m_2$} .
\end{cases}
\end{align*}
Moreover we have
\begin{align*}
\# S_1+ \# S_2 + \# S_3&=
4+\left\lfloor \dfrac{2m_2}{m_1} \right\rfloor -\left\lceil \dfrac{m_2}{m_1} \right\rceil+ \left\lfloor \dfrac{m_2}{m_1} \right\rfloor,\\
&= 3+\left\lfloor \dfrac{2m_2}{m_1} \right\rfloor.
\end{align*}
For the second equality we use the following simple fact:
\[
-\left\lceil x \right\rceil + \left\lfloor x \right\rfloor=-1, \quad \mbox{for any $x \not \in \mathbb Z_+$}.
\]
Hence we obtain
\[
\# I_{m,2}^2 =
\begin{cases}
4+ \left\lfloor \dfrac{2m_2}{m_1} \right\rfloor & \text{if $2m_1 < m_2$},\\
5+ \left\lfloor \dfrac{2m_2}{m_1} \right\rfloor & \text{if $2m_1 > m_2$}.
\end{cases}
\]
Let us put
\[
S=\{ n \in \mathbb Z_+ \text{ such that } m_1+m_2 < n < 2(m_1+m_2) \text{ and } n \not \in I_{m,2}^2\}.
\]
Then we have
\[
\# S = 
\begin{cases}
m_1+m_2-5-\left\lfloor \dfrac{2m_2}{m_1} \right\rfloor & \text{if $2m_1 < m_2$},\\
m_1+m_2-6-\left\lfloor \dfrac{2m_2}{m_1} \right\rfloor & \text{if $2m_1 > m_2$}.
\end{cases}
\]
This completes the proof of the theorem.
\end{proof}
\begin{example}
Let us consider $m_1=5$. Then $m_2=7$ is the unique prime number such that $2m_1 >m_2$.
In this case, $d'=12-6-\lfloor\frac{14}{5}\rfloor=4$. Indeed, the set $S$ is given by $S=\{13, 16, 18, 23\}$.
For the case $2m_1 < m_2$,  we give a list of $d$ and $S$ for the first few prime numbers.
\begin{table}[H]
\begin{center}
\begin{tabular}{|c|c|c|}
\hline
$m_2$ & $d$ & $S$\\
\hline
$11$ & 7 & $\{17, 18, 19, 23, 24, 28, 29\}$ \\
$13$ & 8& $\{19, 21, 22, 24, 27, 29, 32, 34\}$ \\
$17$ & 11 & $\{23, 24, 26, 28, 29, 31, 33, 36, 38, 41, 43\}$\\
$19$ & 12 & $\{26, 27, 28, 31, 32, 33, 36, 37, 41, 42, 46, 47\} $\\
$23$ & 14 & $\{29, 31, 32, 34, 36, 37, 39, 41, 42, 44, 47, 49, 52, 54\}$\\
\hline
\end{tabular}
\caption{Some cases of $d$ and $S$ for $m_1=5$}
\end{center}
\end{table}
\end{example}
In the proof of $S_i \cap S_j= \emptyset$ for $1 \leq i,j \leq 4$, the condition ``$5 \leq m_1$" is only used for $S_1 \cap S_4 = \emptyset$.
As we can see in Example \ref{ex35}, $S_1 \cap S_4 = \emptyset$ is not true in general for $m_1=3$.
Thus we need to deal with this special case separately.\par
Now, suppose that $m_1=3$ and $3m_2 \in S_4$. Putting $r=m_2$, we easily see that $rm_1= 3m_2$ and thus $S_1 \cap S_4=\{m_1m_2 \} $.
Therefore we obtain
\begin{align*}
\# S &= m_1+m_2-5-\left\lfloor \dfrac{2m_2}{m_1} \right\rfloor,\\
&= m_2 -2- \left\lfloor \dfrac{2m_2}{3} \right\rfloor,
\end{align*}
for any prime number such that $m_2 \geq 5$.
In summary we have the following result for $m_1=3$:
\begin{theorem}
Let $m_2$ be a prime numbers such that $5 \leq m_2 $.
Then there exists $d=m_2 -2- \lfloor \frac{2m_2}{3} \rfloor$ numbers $\{a_i\}_{i=1}^d$ in the interval $(3+m_2, 2 (3+m_2) )$ 
such that $(3,m_2, a_i)\in \mathscr L_3$
for any $1 \leq i \leq d$,
\end{theorem}
The table below gives the values of $f(m_2)=m_2 -2- \lfloor \frac{2m_2}{3} \rfloor$ for $5 \leq m_2 \leq 50$.
\begin{table}[H]
\begin{center}
\begin{tabular}{|c|c|c|c|c|c|c|c|c|c|c|c|c|c|c|c|}
\hline
$m_2$ & 5& 7 & 11 & 13 & 17 & 19 & 23 & 29 & 31 &37 &41&43&47\\
\hline 
$f(m_2)$ & 0& 1 & 2& 3&4&5& 6 &8 &9& 11&12&13&14\\
\hline
\end{tabular}
\caption{Values of $f(m_2) $}
\end{center}
\end{table}
For the reader's convenience, we also give $S$ for the first few cases.
\begin{itemize}
\item $S=\emptyset$ if $(m_1,m_2)=(3,5)$,
\item $S=\{11\}$ if $(m_1,m_2)=(3,7)$,
\item $S=\{16,19\}$ if $(m_1,m_2)=(3,11)$,
\item $S=\{17,20, 23\}$ if $(m_1,m_2)=(3,13)$,
\item $S=\{22, 25, 28, 31\}$ if $(m_1,m_2)=(3,17)$,
\item $S=\{23, 26, 29, 32, 35\}$ if $(m_1,m_2)=(3,19)$,
\item $S=\{28, 31, 34, 37, 40, 43\}$ if $(m_1,m_2)=(3,23)$,
\item $S=\{34, 37, 40, 43, 46, 49, 52, 55\}$ if $(m_1,m_2)=(3,29)$.
\end{itemize}
\section*{Acknowledgement}
The author would like to thank Kang-Tae Kim and Feng Rong for their interest and comments on this work.
Moreover the author also thanks to Hyeseon Kim and Van Thu Ninh for their comments 
on an earlier draft of the manuscript which improve the exposition of the paper.
\bibliographystyle{amsplain}

\begin{thebibliography}{10}
\bibitem{Bell} S. Bell, Proper holomorphic mappings between circular domains, Comment. Math. Helv. 57 (4) (1982) 532--538.
\bibitem{Ber} F. Berteloot and G. Patrizio, A Cartan theorem for proper holomorphic mappings of complete circular domains, Adv. Math. 153 (2000) 342--352.
\bibitem{BKU}R. Braun, W. Kaup and H. Upmeier, On the automorphisms of circular and Reinhardt domains in complex Banach spaces, Manuscr. Math. 25 (1978) 97--133. 
\bibitem{IK} H. Ishi and C. Kai, The representative domain of a homogeneous bounded domain, Kyushu J. Math. 64 (1) (2010) 35--47.
\bibitem{IY} H. Ishi and S. Yamaji, Some estimates of the Bergman kernel of minimal bounded homogeneous domains, J. Lie Theory, 21 (2011), 755--769. 
\bibitem{Kaup} W. Kaup, \"{U}ber das {R}andverhalten von holomorphen {A}utomorphismen beschr\"ankter Gebiete, Manuscr. Math., 3 (1970), 257--270. 
\bibitem{KNY}  H. Kim, V. T. Ninh and A. Yamamori, The automorphism group of a certain unbounded  non-hyperbolic domain, J. Math. Anal. Appl. 409 (2014), 637--642.
\bibitem{Kob} S. Kobayashi, Hyperbolic complex spaces. Grundlehren der Mathematischen Wissenschaften, 318, Springer-Verlag, Berlin, 1998.
\bibitem{Mas} M. Maschler, Minimal domains and their Bergman kernel function, Pacific J. Math., 6 (1956), 501--516. 
\bibitem{OK} S. Ozaki and S. Kato, Bergman minimal domains in several complex variables, Trans. Amer. Math. Soc. 162 (1971), 63--69. 
\bibitem{Rong} F. Rong,  On automorphisms of quasi-circular domains fixing the origin, preprint (arXiv:1403.7769).
\bibitem{TW} Z. Tu and L. Wang, Rigidity of proper holomorphic mappings between certain unbounded non-hyperbolic domains, J. Math. Anal. Appl. 419 (2014) 703--714.
\bibitem{Y2013} A. Yamamori, The Bergman kernel of the Fock-Bargmann-Hartogs domain and the polylogarithm function, Complex Var. Elliptic Equ. 58 (2013), no. 6, 783--793.
\bibitem{Yamamori-Bull} A. Yamamori, Automorphisms of normal quasi-circular domains, Bull. Sci. Math., 138 (2014), 406--415.
\bibitem{Tsu} T. Tsuboi, Bergman representative domains and minimal domains, Japan. J. Math. 29 (1959) 141--148.
\end{thebibliography}

\end{document}